\documentclass{amsart}
\usepackage[utf8]{inputenc}
\usepackage[all]{xy}
\usepackage{tikz}
\usepackage{upgreek,amsthm,bbm,bm}
\usepackage{amssymb,amstext,amsmath,amscd,amsthm,hyperref,tikz,tikz-cd,amsfonts,enumerate,graphicx,latexsym,color,geometry}
\usepackage[all]{xy}
\newtheorem{thm}{Theorem}[section]
\newtheorem{cor}[thm]{Corollary}
\newtheorem{lem}[thm]{Lemma}
\newtheorem{prop}[thm]{Proposition}

\newtheorem*{thm*}{Theorem}
\newtheorem*{cor*}{Corollary}
\theoremstyle{definition}
\newtheorem{conv}[thm]{Convention}
\newtheorem{defn}[thm]{Definition}
\newtheorem{rem}[thm]{Remark}
\newtheorem{ques}[thm]{Question}

\newtheorem*{conj*}{Conjecture}
\newtheorem{exam}[thm]{Example}

\newtheorem*{claim*}{Claim}
\newtheorem*{ques*}{Question}

\theoremstyle{remark}
\newtheorem*{ac}{Acknowledgments}
\numberwithin{equation}{thm}
\def\A{\mathcal{A}}
\def\ann{\operatorname{Ann}}
\def\Ass{\operatorname{Ass}}
\def\Assh{\operatorname{Assh}}
\def\B{\mathcal{B}}
\def\C{\mathrm{C}}
\def\c{\mathcal{E}}
\def\cmd{\operatorname{cmd}}
\def\D{\mathrm{D}}
\def\depth{\operatorname{depth}}
\def\E{\mathrm{E}}
\def\e{\mathbb{E}}
\def\Ext{\operatorname{Ext}}
\def\ge{\geqslant}
\def\Gdim{\operatorname{G-dim}}
\def\height{\operatorname{ht}}
\def\H{\mathrm{H}}
\def\Hom{\operatorname{Hom}}
\def\I{\mathrm{I}}
\def\id{\mathrm{id}}
\def\im{\mathrm{Im}}
\def\le{\leqslant}
\def\m{\mathfrak{m}}
\def\Min{\operatorname{Min}}
\def\mod{\operatorname{mod}}

\def\P{\mathrm{P}}
\def\p{\mathfrak{p}}
\def\pd{\operatorname{pd}}
\def\q{\mathfrak{q}}
\def\rhom{\operatorname{\mathbf{R}Hom}}

\def\Supp{\operatorname{Supp}}
\def\T{\mathrm{T}}
\def\t{\mathbb{T}}
\def\Tor{\operatorname{Tor}}
\def\type{\operatorname{type}}

\def\Z{\mathbb{Z}}
\begin{document}
\title{Characteristic modules over a local ring}
\author{Mohsen Gheibi}
\address[Mohsen Gheibi]{Department of Mathematics, Florida A{\&}M University, Tallahassee FL, USA }
\email{mohsen.gheibi@famu.edu}
\author{Ryo Takahashi}
\address[Ryo Takahashi]{Graduate School of Mathematics, Nagoya University, Furocho, Chikusaku, Nagoya 464-8602,
Japan}
\email{takahashi@math.nagoya-u.ac.jp}
\urladdr{https://www.math.nagoya-u.ac.jp/~takahashi/}
\thanks{Takahashi was partly supported by JSPS Grant-in-Aid for Scientific Research 23K03070}
\begin{abstract}
Let $R$ be a commutative noetherian local ring, and let $M$ be a finitely generated $R$-module. Inspired by works of Vasconcelos and Briggs on characterization of complete intersection local rings through the homological properties of the conormal module, in this paper, we define the characteristic module $\T_M$ and the cocharacteristic module $\E_M$ of $M$, and investigate their properties. Our main results include characterizations of Cohen--Macaulay and Gorenstein local rings. Also, we show that if the injective dimension of the conormal module over an almost complete intersection ring is finite, then $R$ is a complete intersection.
\end{abstract}
\maketitle
\section{Introduction}

Let $(R,\m)$ be a commutative noetherian local ring, and
let $\pi:Q\twoheadrightarrow\widehat R$ be a {\em Cohen presentation} of $R$; that is, a surjective homomorphism from a regular local ring $Q$ to the $\m$-adic completion $\widehat R$ of $R$. Let $I$ be the kernel of $\pi$.
Then the {\em conormal module} of $R$ is defined as the $R$-module $\C_R$ such that $\widehat{\C_R}\cong I/I^2$. Equivalently, one may think of the conormal module as the first Tor module of $\widehat R$ over $Q$,
$$
\widehat{\C_R}\cong\Tor_1^Q(\widehat R,\widehat R).
$$
The notion of conormal modules has been studied actively for more than half a century in its algebraic and geometric aspects. For example, see \cite{AH,B,F,V, V2} and the references therein. The most celebrated result at the subject is the Vasconcelos' Conjecture that states that if $\pd_R \C_R<\infty$, then the ring $R$ is a complete intersection. The conjecture recently was resolved by Briggs \cite{B}.
This result inspired us to investigate whether the ring $R$ can be characterized regarding certain (homological) properties of higher Tor modules of $\widehat{R}$ over $Q$.  
In this paper, we focus on the last (nonzero) Tor module of $\widehat{R}$ over $Q$ and define the {\em characteristic module} of $R$ to be such an $R$-module $\T_R$ that satisfies
$$
\widehat{\T_R}\cong\Tor_s^Q(\widehat R,\widehat R),
$$
where $s=\pd_Q\widehat R$.
More generally, for each finitely generated $R$-module $M$ we shall define the {\em characteristic module} $\T_M$ of $M$ by the existence of an isomorphism $\widehat{\T_M}\cong\Tor_s^Q(\widehat R,\widehat M)$. We also consider the dual notion and define the {\em cocharacteristic module} of $M$ as an $R$-module $\E_M$ such that $\widehat{\E_M}\cong\Ext_Q^s(\widehat R,\widehat M)$.

First we discuss the existence and uniqueness of (co)characteristic modules and show that if they exist, they do not depend on the choice of a Cohen presentation. More precisely, we prove the following theorem.

\begin{thm}\label{14}
Let $R$ be a local ring, and $M$ be an $R$-module.
\begin{enumerate}[\rm(1)]
\item
If $R$ is a homomorphic image of a Gorenstein local ring, then both $\T_M$ and $\E_M$ exist.
\item 
If $\T_M$ or $\E_M$ exists, then it is uniquely determined from $M$ up to isomorphism.
\end{enumerate}
\end{thm}

As to the structure of (co)characteristic modules and their completions, we present the theorem below.

\begin{thm}\label{17}
Let $R$ be a local ring.
Then the following statements hold true.
\begin{enumerate}[\rm(1)]
\item
Suppose that $\E_R$ exists.
If $R$ is a Cohen--Macaulay ring, then $\E_R$ is a canonical module of $R$. 
\item
Suppose that $\E_R$ exists.
For every Cohen presentation $Q\twoheadrightarrow\widehat R$, there are isomorphisms 
$$
\E_{\widehat R}\cong\widehat{\E_R}\cong\Ext_Q^s(\widehat R,\widehat R)\cong\Ext_Q^s(\widehat R,Q),\qquad s:=\pd_Q\widehat R=\dim Q-\depth R.
$$ 
\item
If $\E_M$ exists, then $\E_{\widehat M}\cong\widehat{\E_M}\cong\E_{\widehat R}\otimes_{\widehat R}\widehat M$.
If $\T_M$ exists, then $
\T_{\widehat M}\cong\widehat{\T_M}\cong\Hom_{\widehat R}(\E_{\widehat R},\widehat M)$.
\end{enumerate}
\end{thm}

Cohen--Macaulay local rings are characterized in terms of the dimensions of (co)characteristic modules.

\begin{thm}\label{9}
Suppose that $R$ is a homomorphic image of a Gorenstein local ring.
The following conditions are equivalent.
\begin{enumerate}[\rm(1)]
\item 
The ring $R$ is Cohen--Macaulay.
\item 
One has $\dim\E_R=\dim R$.
\item 
One has $\dim\T_R=\dim R$.
\item
There exists an $R$-module $M$ such that $\dim\E_M=\dim R$.
\item
There exists an $R$-module $M$ such that $\dim\T_M=\dim R$.
\end{enumerate}
\end{thm}

As an application of Theorem \ref{9}, we can also relate a (co)characteristic module to the generic Gorensteinness, and obtain a characterization of the Gorenstein local ring.

\begin{thm}\label{10}
Let $R$ be a local ring. The following conditions are equivalent.
\begin{enumerate}[\rm(1)]
\item
The ring $R$ is Gorenstein.
\item One has that $\E_R$ exists and is nonzero and free.
\item One has that $\T_R$ exists and is nonzero and free.
\item There is  an $R$-module $M$ with $\dim M=\dim R$ such that $\E_M$ exists and $\E_M\cong M$.
\item There is an $R$-module $M$ with $\dim M=\dim R$ and $\Gdim M<\infty$ such that $\T_M$ exists and $\T_M\cong M$.
\item There is an $R$-module $M$ with $\Gdim_R M<\infty$ such that $\E_M$ exists, $\pd_R\E_M<\infty$ and $\dim\E_M=\dim R$.
\item There is an $R$-module $M$ with $\pd_R M<\infty$ such that $\T_M$ exists, $\Gdim_R\T_M<\infty$ and $\dim\T_M=\dim R$.
\end{enumerate}
\end{thm}

The organization of this paper is as follows. In Section 2, we collect definitions and properties of basic notions used in this paper. In Section 3, we state the precise definitions of a quasi-canonical module of $R$, a characteristic module and a cocharacteristic module of a given $R$-module. We investigate their basic properties in this section. In Section 4, we study characteristic and cocharacteristic modules over a quotient of a regular local ring from functorial approaches. In the final Section 5, we give the proofs of the theorems mentioned above and state some other related results.

\section{Preliminaries}

This section is devoted to preliminaries for the later section. We begin with stating our convention adopted throughout this paper.

\begin{conv}
Throughout this paper, we assume that all rings are commutative and noetherian. For a local ring $(R,\m)$ we denote by $\widehat{(-)}$ the $\m$-adic completion.
A {\em Cohen presentation} is by definition a surjective ring homomorphism $Q\twoheadrightarrow\widehat R$ where $Q$ is a regular local ring.
\end{conv}

We recall the definition of a dualizing complex of $R$ and some of its fundamental properties.

\begin{defn}
\begin{enumerate}[\rm(1)]
\item
Let $D$ be an $R$-complex with finitely generated homology (note then that $D$ is homologically bounded).
We say that $D$ is {\em dualizing} if it has finite injective dimension and the homothety morphism $R\to\rhom_R(D,D)$ is a quasi-isomorphism.
\item
We denote by $\D_R$ a {\em normalized dualizing complex} of $R$, that is, $\D_R$ is a dualizing $R$-complex with $\inf\{i\mid\H^i(\D_R)\ne0\}=0$. 
\end{enumerate}
\end{defn}

\begin{rem}
\begin{enumerate}[\rm(1)]
\item
A dualizing complex of $R$ is uniquely determined up to shifts in the derived category of $R$; see \cite[Chapter V, Theorem 3.1]{H} and \cite[(A.8.3)]{C}.
Hence a normalized dualizing complex $\D_R$ of $R$ is uniquely determined up to quasi-isomorphism.
\item 
A local ring $R$ admits a dualizing complex if and only if $R$ is a homomorphic image of a Gorenstein local ring; see \cite[Corollary 1.4]{K}.
\item 
If $R$ is a homomorphic image of a Gorenstein local ring $S$, then 
$$
\D_R\cong\rhom_S(R,S)[\dim S-\dim R]
$$
in the derived category of $R$.
This is a consequence of the derived Hom-$\otimes$ adjointness and the fact that $S$ is a dualizing complex of $S$.
\item
If $R$ is Cohen--Macaulay, then $\H^0(\D_R)$ is a canonical module of $R$ by (3) and \cite[Theorem 3.3.7(b)]{BH}.
\end{enumerate}
\end{rem}

Next we recall some basic numerical invariants for local rings and modules over them.

\begin{defn}
Let $(R,\m,k)$ be a local ring.
\begin{enumerate}[(1)]
\item
We denote by $\cmd R$ the {\em Cohen--Macaulay defect} of $R$, that is, $\cmd(R)=\dim R-\depth R$.
\item 
We denote by $\type_R(M)$ the {\em type} of an $R$-module $M$, that is, $\type_R(M)=\dim_k\Ext_R^{\depth_RM}(k,M)$.
\item
For a finitely generated $R$-module $M$ we denote by $\nu_R(M)$ the minimal number of generators of $M$, that is, $\nu_R(M)=\dim_k(M\otimes_Rk)$.
\end{enumerate}
\end{defn}

The Auslander class and the Bass class of a Cohen--Macaulay local ring with a canonical module, and the Poincar\'e and the Bass series of a complex play an important role in Section 4.

\begin{defn}\label{AB}
Let $R$ be a Cohen--Macaulay local ring with a canonical module $\omega$.
The {\em Auslander class} $\A(R)$ is defined as the class of $R$-modules $M$ such that the natural map $M\to\Hom_R(\omega,\omega\otimes_RM)$ is an isomorphism, $\Tor_{>0}^R(\omega,M)=0$ and $\Ext_R^{>0}(\omega,\omega\otimes_RM)=0$.
The {\em Bass class} $\B(R)$ is defined as the class of $R$-modules $M$ such that the natural map $\omega\otimes_R\Hom_R(\omega,M)\to M$ is an isomorphism, $\Ext_R^{>0}(\omega,M)=0$ and $\Tor^R_{>0}(\omega,\Hom_R(\omega,M))=0$.
\end{defn}

\begin{defn}
Let $X$ be an $R$-complex with finitely generated homology. 
Then the formal Laurent series 
$$
\P\!_X(t)=\sum_{n\in\Z}\beta_n^R(X)\,t^n,\qquad
\I^X(t)=\sum_{n\in\Z}\mu^n_R(X)\,t^n
$$
are respectively called the {\em Poincar\'e series} and the {\em Bass series} of $X$, where $\beta_n^R(X)=\dim_k\Tor_n^R(k,X)$ and $\mu^n_R(X)=\dim_k\Ext^n_R(k,X)$ are the $n$th Betti and Bass numbers of $X$, respectively.
\end{defn}

\section{Basic properties of characteristic modules}

In this section, we state basic properties of characteristic modules and cocharacteristic modules. First of all, we define a quasi-canonical module of a local ring.

\begin{defn}
We define a {\em quasi-canonical module} of $R$ to be an $R$-module $\c_R$ such that
$$
\widehat{\c_R}\cong\Ext_Q^{\dim Q-\depth R}(\widehat R,Q),
$$
where $Q\twoheadrightarrow\widehat R$ is a Cohen presentation.
\end{defn}

The following proposition guarantees the existence and uniqueness of a quasi-canonical module of a local ring, and explains that it is a generalization of a canonical module of a Cohen--Macaulay local ring.

\begin{prop}\label{15}
\begin{enumerate}[\rm(1)]
\item
Suppose that $R$ is a homomorphic image of a Gorenstein local ring.
Then a quasi-canonical module of $R$ exists.
Indeed, $\H^{\cmd R}(\D_R)$ is a quasi-canonical module of $R$.
\item
If a quasi-canonical module of $R$ exists, then it is uniquely determined up to isomorphism.
\item 
If $R$ is Cohen--Macaulay, then a quasi-canonical module of $R$ is a canonical module of $R$.
\end{enumerate}
\end{prop}

\begin{proof}
(1) As $R$ and $\widehat R$ are homomorphic images of Gorenstein local rings, there exist normalized dualizing complex $\D_R$ and $\D_{\widehat R}$, and we easily see that $\widehat{\D_R}\cong\D_{\widehat R}$.
Set $H=\H^{\cmd R}(\D_R)$.
Let $Q\twoheadrightarrow\widehat R$ be a Cohen presentation.
We have
$$
\widehat H\cong\H^{\cmd R}(\widehat{\D_R})\cong\H^{\cmd R}(\D_{\widehat R})\cong\Ext_Q^{\cmd R+(\dim Q-\dim\widehat R)}(\widehat R,Q)=\Ext_Q^{\dim Q-\depth R}(\widehat R,Q).
$$
Therefore, $H$ is a quasi-canonical module of $R$.

(2) Let $X$ and $Y$ be quasi-canonical modules of $R$.
Then there exist Cohen presentations $Q\twoheadrightarrow\widehat R$ and $Q'\twoheadrightarrow\widehat R$ such that $\widehat X\cong\Ext_Q^{\dim Q-\depth R}(\widehat R,Q)$ and $\widehat Y\cong\Ext_{Q'}^{\dim Q'-\depth R}(\widehat R,Q')$.
Since $\widehat R$ is a homomorphic image of the Gorenstein local rings $Q$ and $Q'$, a normalized dualizing complex $\D_{\widehat R}$ of $\widehat R$ exists, and we have
$$
\rhom_Q(\widehat R,Q)[\dim Q-\dim\widehat R]\cong\D_{\widehat R}\cong\rhom_{Q'}(\widehat R,Q')[\dim Q'-\dim\widehat R]
$$
in the derived category of $\widehat R$.
It follows that
\begin{align*}
\widehat X&\cong\Ext_Q^{\dim Q-\depth R}(\widehat R,Q)
=\Ext_Q^{\cmd\widehat R+(\dim Q-\dim\widehat R)}(\widehat R,Q)
\cong\H^{\cmd\widehat R}(\D_{\widehat R}),\\
\widehat Y&\cong\Ext_{Q'}^{\dim Q'-\depth R}(\widehat R,Q')
=\Ext_{Q'}^{\cmd\widehat R+(\dim Q'-\dim\widehat R)}(\widehat R,Q')
\cong\H^{\cmd\widehat R}(\D_{\widehat R}).
\end{align*}
We thus obtain $\widehat X\cong\widehat Y$, which implies $X\cong Y$.

(3) This is an immediate consequence of \cite[Theorem 3.3.7(b)]{BH}.
\end{proof}

A quasi-canonical module is compatible with completion.

\begin{cor}\label{16}
If $\c_R$ exists, then there is an isomorphism $\widehat{\c_R}\cong\c_{\widehat R}$.
\end{cor}

\begin{proof}
Let $Q\twoheadrightarrow\widehat R$ be a Cohen presentation.
Then $\D_{\widehat R}=\rhom_Q(\widehat R,Q)[\dim Q-\dim\widehat R]$, and $\c_{\widehat R}=\H^{\cmd\widehat R}(\D_{\widehat R})=\Ext_Q^{\cmd\widehat R+\dim Q-\dim\widehat R}(\widehat R,Q)=\Ext_Q^{\dim Q-\depth R}(\widehat R,Q)\cong\widehat{\c_R}$ by Proposition \ref{15}(1).
\end{proof}

Now we define a characteristic module and a cocharacteristic module of a given $R$-module.

\begin{defn}\label{def1}
Let $M$ be a finitely generated $R$-module. We define a {\em characteristic module} of $M$ to be an $R$-module $\T_M$ such that 
$$
\widehat{\T_M} \cong \Tor^Q_{\dim Q-\depth R}(\widehat{R},\widehat{M})
$$ where $Q\twoheadrightarrow \widehat{R}$ is a Cohen presentation of $R$.
Dually, we define a {\em cocharacteristic module} of $M$ to be an $R$-module $\E_M$ such that 
$$
\widehat{\E_M} \cong \Ext_Q^{\dim Q-\depth R}(\widehat{R},\widehat{M}).
$$
\end{defn}

The definition of characteristic modules is motivated by the notion of the conormal module of a local ring.

\begin{rem}
Let $R$ be a local ring.
The {\em conormal module} of $R$ is by definition an $R$-module $\C_R$ such that $\widehat{\C_R}\cong I/I^2$, where $I$ is the kernel of a minimal Cohen presentation $Q\twoheadrightarrow\widehat R$.
Hence
$$
\widehat{\C_R}\cong\Tor_1^Q(\widehat R,\widehat R).
$$
One has $\pd_Q\widehat R=n-t$, where $n=\dim Q$ and $t=\depth R$.
The minimal free resolution of $\widehat R$ over $Q$ is of the form $F=(0\to F_{n-t}\to F_{n-t-1}\to\cdots\to F_0\to0)$.
There is a complex
$$
F\otimes_Q\widehat R=(0\to F_{n-t}\otimes_Q\widehat R\to F_{n-t-1}\otimes_Q\widehat R\to\cdots\to F_0\otimes_Q\widehat R\to0).
$$
The completion of $\C_R$ is the first homology of this complex, while the completion of $\T_R$ is the left-end homology of this complex.
\end{rem}

Spectral sequence arguments give rise to equalities of types and minimal numbers of generators.

\begin{prop}\label{type}
Let $(R,\m,k)$ be a local ring of depth $t$, and let $M$ be a finitely generated $R$-module. Assume $\T_M$ and $\E_M$ exist. Let $Q\twoheadrightarrow \widehat{R}$ be a Cohen presentation, where $Q$ is a regular local ring of dimension $n$.
\begin{enumerate}[\rm(1)]
\item One has $\nu_R(\E_M)=\type(R)\nu_R(M)$.
\item Assume $\pd_QM=\pd_QR$ and $\depth_{\widehat{R}} \Tor^Q_{n-t-i}(\widehat{R},\widehat{M})\ge t-i+1$ for all $1\le i \le n-t$. Then one has $\depth_R \T_M = \depth R$ and $$\type_R(\T_M)=\type(R)\type_R(M).$$
\end{enumerate} 
\end{prop}
\begin{proof} Passing to the completion, we may assume $R=\widehat{R}$. Let $F$ be a free resolution of $M$ over $Q$, $U$ a free resolution of $k$ over $R$, $J'$ an injective resolution of $M$ over $Q$, and $J$ an injective resolution of $R$ over itself.

(1) The double complex $U\otimes_R\Hom_Q(R,J')$ induces an spectral sequence $$\E^2_{p,q}=\Tor^R_p(k,\Ext^q_Q(R,M))\Longrightarrow \H_{p-q}(U\otimes_R\Hom_Q(R,J')).$$
Since $Q$ is regular, there exists an isomorphism of complexes $U\otimes_R\Hom_Q(R,J')\cong \Hom_Q(\Hom_R(U,R),J')$ from which we get the following isomorphisms 
\[\begin{array}{rl}\
\H_m(U\otimes_R\Hom_Q(R,J')) &\cong \H_m(\Hom_Q(\Hom_R(U,R),J'))\\
&\cong \H_m(\Hom_Q(\Hom_R(k,J),J'))\\
&\cong \H_m(\Hom_R(k,J)^{\vee}\otimes_k \Hom_Q(k,J'))\\
&\cong \oplus_{i-j=m}\Ext^i_R(k,R)^{\vee}\otimes_k \Ext^j_Q(k,M),
\end{array}\]
where $(-)^\vee:=\Hom_k(-,k)$. Since $\pd_QR=n-t$ and the maps $d^r_{p,q}$ on $\E^r$ page are bidegree $(-r,-r+1)$, we have $\E^2_{0,n-t}\cong \E^\infty_{0,n-t}$. Therefore $k\otimes_R \E_M \cong \Ext^t_R(k,R)^{\vee}\otimes_k \Ext^n_Q(k,M)$. Since $\Ext^n_Q(k,M) \cong k\otimes_Q M$, we are done.  

(2) The double complex $\Hom_R(U,R\otimes_Q F)$ induces an spectral sequence $$\E^2_{p,q}\cong\Ext^p_R(k,\Tor^Q_q(R,M))\Longrightarrow \H^{p-q}(\Hom_R(U,R\otimes_Q F)).$$
    Since $F$ is a perfect complex over $Q$, there is a natural isomorphism of complexes $\Hom_R(U,R)\otimes_Q F \cong \Hom_R(U,R\otimes_Q F)$. Therefore one has
    \[\begin{array}{rl}\
\H^m(\Hom_R(U,R\otimes_Q F)) &\cong \H^m(\Hom_R(U,R)\otimes_Q F)\\
&\cong \H^m(\Hom_R(k,J)\otimes_Q F)\\
&\cong \H^m(\Hom_R(k,J)\otimes_k (k\otimes_Q F))\\
&\cong \oplus_{i-j=m}\Ext^i_R(k,R)\otimes_k \Tor^Q_j(k,M).
\end{array}\] Since the maps $d^r_{p,q}$ on $\E^r$ page are bidegree $(r,r-1)$, the spectral sequence and the assumptions show that $\Ext^j_R(k,\T_M)=0$ for all $0\le j \le t-1$, and $\E^2_{t,n-t}\cong \E^{\infty}_{t,n-t} \cong \Ext^t_R(k,R)\otimes_k\Tor^Q_{n-t}(k,M)$. Therefore $\depth_R\T_M=t$ and since $Q$ is regular, we have $\Tor^Q_{n-t}(k,M)\cong \Ext^t_Q(k,M)$. Let $\underline{x}:=x_1,\dots, x_t \in Q$ be a regular sequence on both $R$ and $M$. We abuse notations and write $\underline{x}$  for its image in $R$ as well. Then one has $\Ext^t_Q(k,M)\cong \Hom_{Q/(\underline{x})}(k,M/(\underline{x})M)=\Hom_{R/(\underline{x})}(k,M/(\underline{x})M)\cong \Ext^t_R(k,M)$. This finishes the proof.
\end{proof} 

The following example shows that the equality of Proposition \ref{type}(2) may fail if the depth assumption on Tor modules is removed. We used Macaulay2 for computations. 

\begin{exam}
Let $R$ be the completion of the third Veronese subring of the polynomial ring $k[s,t]$ over a field $k$, that is, $R=k[\![s^3,s^2t,st^2,t^3]\!]$. Then $R$ is a Cohen--Macaulay complete local ring of dimension $2$. We have $R\cong Q/I$ where $Q=k[\![w,x,y,z]\!]$ and $I=(x^2-yw,y^2-xz,xy-wz)$. One has $\type(R)=2$ (see \cite[Exercise 3.6.21]{BH}) but $\type_R(\T_R)=3\neq \type(R)^2=4$. Moreover, the module $\T_R=\Tor^Q_2(R,R)$ is maximal Cohen--Macaulay and $\depth_R \Tor^Q_1(R,R)=1$. Since $\dim \Tor^Q_1(R,R) = \dim I/I^2 =2$, the conormal module of $R$ is not Cohen--Macaulay.
\end{exam}

One obtains a sufficient condition for an artinian local ring $R$ to be Gorenstein.

\begin{cor}\label{artinian}
Let $R$ be an artinian local ring. If there exists a nonzero $R$-module $M$ such that $\type_R(M)\ge \type_R(\T_M)$, then $R$ is Gorenstein. In particular, if $M\cong \T_M$, then $R$ is Gorenstein.
\end{cor}

The following example shows that Corollary \ref{artinian} is not true if $R$ is not Cohen--Macaulay. 

\begin{exam}\label{e2}
    Let $k$ be a field, $Q=k[[x,y]]$,  $I=(x^2,xy)$, and $R=Q/I$. Then $R$ is a one-dimensional complete local ring and $\depth R =0$. One has $\type (R)=1$ and $\T_k\cong k$ but $R$ is not Gorenstein. Also for $M=R/(x)$, one has $\pd_QM=1$ and therefore $\T_M=\Tor^Q_2(R,M)=0$. This shows the assumption $\pd_QM=\pd_QR$ is necessary in Proposition \ref{type}.
\end{exam}

\section{Functorial approaches to characteristic modules and applications}

In this section, we restrict to the case of a quotient of a regular local ring and discuss functorial aspects of characteristic modules and cocharacteristic modules over it.
Throughout this section, let $Q$ be an $n$-dimensional regular local ring, $I$ an ideal of $Q$, $R=Q/I$ and $t=\depth R$.
Then $\pd_QR=n-t$.
Let $F=(0\to F_{n-t}\xrightarrow{\partial_{n-t}}\cdots\xrightarrow{\partial_1}F_0\to0)$ be a minimal free resolution of $R$ over $Q$.
Then $\D_R=\rhom_Q(R,Q)[n-\dim R]$ and $E:=\c_R=\H^{\cmd R}(\D_R)=\Ext_Q^{n-t}(R,Q)$ by Proposition \ref{15}(1).

For a finitely generated $R$-module $M$ we define the $R$-modules $\t(M)$ and $\e(M)$, which play a main role in this section.

\begin{defn}
For a finitely generated $R$-module $M$ we define the modules $\t(M)$ and $\e(M)$ as follows.
$$
\t(M)=\Tor_{n-t}^Q(R,M),\qquad\e(M)=\Ext^{n-t}_Q(R,M).
$$
\end{defn}

The modules $\t(M)$ and $\e(M)$ can be described by using the quasi-canonical module $E$.

\begin{prop}\label{2}
There are isomorphisms of covariant functors from $\mod R$ to $\mod R$:
$$
\t(-)\cong\Hom_R(E,-),\qquad
\e(-)\cong E\otimes_R-.
$$
In particular, one has an adjoint pair 
$$
\xymatrix{\mod R\ar@<2mm>[rrr]^\e_\perp&&& \mod R\ar@<2mm>[lll]^\t}
$$
and an $R$-module isomorphism $\e(R)\cong E$.
\end{prop}

\begin{proof}
Let $M$ be a finitely generated $R$-module.
By definition, the module $\t(M)$ is the kernel of the map $\partial_{n-t}\otimes_QM$, while the modules $E$ and $\e(M)$ are the cokernels of the maps $\Hom_Q(\partial_{n-t},Q)$ and $\Hom_Q(\partial_{n-t},M)$, respectively.
There is a commutative diagram
$$
\xymatrix{
0\ar[r]& \t(M)\ar[r]& F_{n-t}\otimes_QM\ar[rrr]^{\partial_{n-t}\otimes_QM}&&& F_{n-t-1}\otimes_QM\\
0\ar[r]& \Hom_Q(E,M)\ar[r]& \Hom_Q(\Hom_Q(F_{n-t},Q),M)\ar[rrr]^{\Hom_Q(\Hom_Q(\partial_{n-t},Q),M)}\ar[u]^\cong&&& \Hom_Q(\Hom_Q(F_{n-t-1},Q),M)\ar[u]^\cong
}
$$
of $R$-modules, where the vertical maps are natural isomorphisms.
We get an isomorphism $\Hom_Q(E,M)\to\t(M)$.
As $E$ and $M$ are $R$-modules, we have $\Hom_Q(E,M)=\Hom_R(E,M)$.
Hence $\t(M)\cong\Hom_R(E,M)$.
There is also a commutative diagram
$$
\xymatrix{
\Hom_Q(F_{n-t-1},M)\ar[rrr]^{\Hom_Q(\partial_{n-t},M)}&&& \Hom_Q(F_{n-t},M)\ar[r]& \e(M)\ar[r]& 0\\
\Hom_Q(F_{n-t-1},Q)\otimes_QM\ar[rrr]^{\Hom_Q(\partial_{n-t},Q)\otimes_QM}\ar[u]^\cong&&& \Hom_Q(F_{n-t},Q)\otimes_QM\ar[r]\ar[u]^\cong& E\otimes_QM\ar[r]& 0
}
$$
of $R$-modules, where the vertical maps are natural isomorphisms.
We get an isomorphism $E\otimes_QM\to\e(M)$.
As $E$ and $M$ are $R$-modules, we have $E\otimes_QM\cong E\otimes_QR\otimes_RM\cong E\otimes_RM$.
Hence $\e(M)\cong E\otimes_RM$.
\end{proof}

Next we define the canonical homomorphisms $\alpha_M$ and $\beta_M$ for each $R$-module $M$.

\begin{defn}
Let $M$ be a finitely generated $R$-module.
Define the maps
$$
\alpha_M:M\to\Hom_R(E,E\otimes_RM),\qquad
\beta_M:E\otimes_R\Hom_R(E,M)\to M
$$
by $\alpha_M(m)(e)=e\otimes m$ and $\beta_M(e\otimes f)=f(e)$ for $e\in E$, $m\in M$ and $f\in\Hom_R(E,M)$.
Note that $\alpha$ and $\beta$ are just unit and counit of the adjunction established in Proposition \ref{2};
$$
\alpha_M:M\to\t\e(M),\qquad\beta_M:\e\t(M)\to M.
$$
\end{defn}

The homomorphisms $\alpha_{\t(M)}$ and $\beta_{\e(M)}$ split.

\begin{prop}
\begin{enumerate}[\rm(1)]
\item
There is an equality $\t(\beta_M)\circ\alpha_{\t(M)}=1$.
Hence $\alpha_{\t(M)}:\t(M)\to\t\e\t(M)$ is a split monomorphism.
\item
There is an equality $\beta_{\e(M)}\circ\e(\alpha_M)=1$.
Hence $\beta_{\e(M)}:\e\t\e(M)\to\e(M)$ is a split epimorphism.
\end{enumerate}
\end{prop}

Here we remark when $\alpha_M$ and $\beta_M$ are isomorphisms.

\begin{rem}\label{r2}
Assume that $R$ is Cohen--Macaulay. In view of \cite[(3.4.11)]{C}, one has that if $\pd_RM<\infty$, then $\alpha_M$ is an isomorphism, and if $\id_RM<\infty$, then $\beta_M$ is an isomorphism.
\end{rem}

The following corollary provides criteria for detecting infinite injective dimension over a local ring by using minimal numbers of generators.

\begin{cor}\label{id}
For a finitely generated $R$-module $M$, if $\id_RM<\infty$, then $\nu_R(M)=\type(R)\nu_R(\t(M))$. In particular, if $\nu_R(\t(M))>\nu_R(M)$, then one has $\id_RM= \infty$. 
\end{cor}
\begin{proof}
   Assume that $M$ is nonzero and $\id_RM<\infty$. Then $R$ is a Cohen--Macaulay ring. By using Proposition \ref{type}(1) and Remark \ref{r2}, we get the desired equality.
\end{proof}

\begin{ques}\label{q1}
    Is it true that if $\id_RI/I^2 < \infty$, then $R$ is a complete intersection ring? Equivalently, if the ring $R$ is not Gorenstein, then is $\id_R(I/I^2)=\infty$?
\end{ques}

The following result provides an affirmative answer for Question \ref{q1} in the case when $R$ is an almost complete intersection. Recall that $R$ is called an {\em almost complete intersection} ring if $\nu(I)\le \height_Q(I)+1$.

\begin{thm}\label{id}
    Let $R$ be an almost complete intersection. If $\id_R(I/I^2)<\infty$, then $R$ is a complete intersection.
\end{thm}
\begin{proof}
    Assume that $\id_RI/I^2<\infty$. If $\nu(I)=\height_Q(I)$, we have nothing to prove. Assume $\nu(I)=\height_Q(I)+1$. We show that this case is not an option.
    By using the Theorem of Bass, we may assume $R$ is Cohen-Macaulay with a canonical module $\omega_R$.
    There exists an exact sequence 
    $$\H_1(I) \overset{g}\to F \to I/I^2 \to 0,$$ 
    where $F$ is a free $R$-module of rank $\nu(I)$. By \cite[Exercise 3.3.28(c)]{BH}, there exists a minimal covering $\oplus^a\omega_R\twoheadrightarrow I/I^2$. Then we get the pullback diagram of $R$-modules   
    $$
\begin{tikzcd}
& & 0 \arrow[d] & 0 \arrow[d] & &\\
& & L \arrow[d] \arrow[r,equal] & L \arrow[d] & & \\
0 \arrow[r] & \im(g) \arrow[r] \arrow[d, equal] & X \arrow[r] \arrow[d] & \oplus^a\omega_R \arrow[r] \arrow[d] & 0 \\
0 \arrow[r] & \im(g) \arrow[r] & F \arrow[r] \arrow[d] & I/I^2 \arrow[r] \arrow[d] & 0\\
& & 0 & 0 &
\end{tikzcd}$$
with exact rows and columns, where $L$ is an $R$-module of finite injective dimension. Since $F$ is free, the middle column splits, and so that $X\cong F\oplus L$. By gluing the top row back to the last exact sequence, we get $$\H_1(I) \to F\oplus L \to \oplus^a\omega_R \to 0.$$ This gives an inequality of multiplicities $a e(\omega_R) + e(\H_1(I))\geq e(F) + e(L) \geq e(F)$; see \cite[Theorem 14.6]{M2}. Note that $e(\omega_R)=e(R)$, $F= \oplus^{\nu(I)}R$, and $\H_1(I)\cong \omega_R$ by \cite[Proposition 2.1]{Kunz}. Hence, one has $a+1\geq \nu(I)$. By Corollary \ref{id}, $\nu(I)=\nu(I/I^2)=\type(R)\nu(\t(I/I^2))$. Since the covering $\oplus^a\omega_R\twoheadrightarrow I/I^2$ is minimal, by Proposition \ref{2} and the fact that $E=\omega_R$, one checks that there is a minimal covering $\oplus^a R\twoheadrightarrow \t(I/I^2)$. Hence,  $\nu(\t(I/I^2))=a$, and we have $a+1\geq \type(R)a$. Since $R$ is not Gorenstein by \cite{Kunz}, we have $\type(R)>1$. The only possibility is that $a=1$, $\type(R)=2$, and equality holds. Therefore, $g$ must be injective. Thus, the exact sequence $0\to \H_1(I) \overset{g}\to F \to I/I^2 \to 0$ implies that $\id(R)<\infty$, a contradiction.
\end{proof}

We remark that Theorem \ref{id} was established in \cite[Observation 5.12]{A} for the case that $I$ is a prime ideal.
The proposition below says about the compatibility of a quasi-canonical module with localization.

\begin{prop}\label{3}
Let $\p\in\Supp_R(\c_R)$, and write $\p=\q/I$ where $\q$ is a prime ideal of $Q$ containing $I$.
One then has $\height\q-\depth R_\p=n-t$ and there is an isomorphism
$$
(\c_R)_\p\cong\c_{R_\p}.
$$
\end{prop}

\begin{proof}
We have $0\ne(\c_R)_\p=\Ext_Q^{n-t}(R,Q)_\p\cong\Ext_{Q_\q}^{n-t}(R_\p,Q_\q)$.
Since $\pd_{Q_\q}R_\p\le\pd_QR=n-t$, we have $\Ext_{Q_\q}^{>(n-t)}(R_\p,Q_\q)=0$.
As $Q_\q$ is a Gorenstein local ring, by using \cite[Exercises 3.1.24 and 3.1.25]{BH}, it is seen that $\dim Q_\q-\depth R_\p=n-t=\dim Q-\depth R$.
Thus $\c_{R_\p}\cong\Ext_{Q_\q}^{n-t}(R_\p,Q_\q)\cong(\c_R)_\p$.
\end{proof}

Now we state and prove the main result of this section.

\begin{thm}\label{8}
The following conditions are equivalent.
\begin{enumerate}[\rm(1)]
\item
There exists a finitely generated $R$-module $M$ such that $\dim\t(M)=\dim R$.
\item
There exists a finitely generated $R$-module $M$ such that $\dim\e(M)=\dim R$.
\item 
One has $\dim\t(R)=\dim R$.
\item 
One has $\dim\e(R)=\dim R$.
\item 
The local ring $R$ is Cohen--Macaulay.
\item 
The map $\alpha_R$ is an isomorphism. 
\item
The map $\beta_E$ is an isomorphism.
\end{enumerate}
\end{thm}

\begin{proof}
First of all, note by Proposition \ref{2} that $\e(R)\cong E=\c_R$.
The implications (4) $\Rightarrow$ (2) and (3) $\Rightarrow$ (1) are obvious.
If $R$ is Cohen--Macaulay, then $E$ is a canonical module of $R$ by Proposition \ref{15}(3). Hence the implication (5) $\Rightarrow$ (4) holds.
Note that the supports of $\Hom_R(E,M)$ and $E\otimes_RM$ are contained in the support of $E$ for each $R$-module $M$.
Using Proposition \ref{2}, we see that the implications (1) $\Rightarrow$ (4) $\Leftarrow$ (2) hold.

Assume (4). The equality $\dim E=\dim R$ implies that $\Supp E\cap\Assh R$ is nonempty, where $\Assh R$ is the set of all primes $\p \in \Ass R$ with $\dim R/\p=\dim R$.
We have
$$
\emptyset\ne\Supp E\cap\Assh R\subseteq\Supp E\cap\Ass R=\Ass\Hom_R(E,R)=\Ass\t(R),
$$
where the last equality follows from Proposition \ref{2}.
We observe that $\dim\t(R)=\dim R$.
Thus (3) follows.
There exists $\p\in\Supp E\cap\Assh R$.
Take a prime ideal $\q$ of $Q$ that contains $I$ and satisfies $\p=\q/I$.
Proposition \ref{3} implies $\height\q-\depth R_\p=n-t$.
As $\p\in\Assh R$, we have $\depth R_\p=0$ and $\dim R/\p=\dim R$.
Hence $\depth R=t=n-\height\q=\dim Q/\q=\dim R/\p=\dim R$.
Therefore, $R$ is Cohen--Macaulay. Thus (5) follows.
We have got the implications (3) $\Leftarrow$ (4) $\Rightarrow$ (5), and now the conditions (1)--(5) are equivalent.

The implications (7) $\Leftarrow$ (5) $\Rightarrow$ (6) follow from the fact that if $R$ is Cohen--Macaulay, then $E$ is a canonical module of $R$. If (6) holds, then $R\cong \t\e(R)$ and therefore $\dim_R\t\e(R)=\dim R$. By setting $M=\e(R)$, we have condition (2) holds.
In a similar way, we can show that (7) implies (1).
\end{proof}

\begin{rem}
The implication (4) $\Rightarrow$ (5) in the above theorem can also be deduced by using \cite[Theorem 8.1.1(b)]{BH}. Indeed, we have $\dim E=\dim (R/\ann_R\Ext_Q^{n-t}(R,Q))\le t=\depth R$. Hence, if $\dim E=\dim R$, then $R$ is Cohen--Macaulay.
\end{rem}

Applying the above theorem, we get the following two corollaries, which give criteria for Gorensteinness in terms of characteristic and cocharacteristic modules. For definition and properties of Gorenstein dimension, we refer the reader to \cite{C}.

\begin{cor}\label{11}
The following conditions are equivalent.
\begin{enumerate}[\rm(1)]
\item
There exists a finitely generated $R$-module $M$ such that $\dim M=\dim R$, $\Gdim M<\infty$ and $M\cong\t(M)$.
\item
There exists a finitely generated $R$-module $M$ such that $\dim M=\dim R$ and $M\cong\e(M)$.
\item
There exists a finitely generated $R$-module $M$ such that $\pd_RM<\infty$, $\Gdim_R\t(M)<\infty$ and $\dim\t(M)=\dim R$.
\item There exists a finitely generated $R$-module $M$ with $\Gdim_RM<\infty$ such that $\pd_R\e(M)<\infty$ and $\dim\e(M)=\dim R$.
\item The $R$-module $\t(R)$ is nonzero and free.
\item The $R$-module $\e(R)$ is nonzero and free.
\item The ring $R$ is Gorenstein.
\end{enumerate}
\end{cor}

\begin{proof}
It follows from Theorem \ref{8} and Propositions \ref{15}(3), \ref{2} that under any of those conditions which are given in the corollary, $R$ is a Cohen--Macaulay local ring and $E\cong\e(R)$ is a canonical module of $R$.

(2) $\Rightarrow$ (7):
Theorem \ref{8} implies that $R$ is Cohen--Macaulay, while $\type R=1$ by Proposition \ref{type}(1).

(7) $\Rightarrow$ (1) and (2):
As $R$ is Gorenstein, Proposition \ref{15}(3) implies $E\cong R$.
We obtain $\e(R)\cong E\cong R$ and  $\t(R)\cong\Hom_R(R,R)\cong R$ by Proposition \ref{2}.
Thus, letting $M=R$, we are done.

(1) $\Rightarrow$ (7):
Theorem \ref{8} and Proposition \ref{15}(3) imply that $R$ is Cohen--Macaulay and $E$ is a canonical module of $R$.
By Proposition \ref{2} and \cite[(3.1.11)]{C} we have $\Hom_R(E,M)\cong\t(M)\cong M\in\A(R)$.
We get $M\in\B(R)$ by \cite[Theorem 2.8(a)]{cdim}, which implies that $\Ext_R^{>0}(E,M)=0$.
It follows that $M\cong\rhom_R(E,M)$ in the derived category of $R$.
Using \cite[(A.7.7)]{C}, we get $\I^M(t)=\I^{\rhom_R(E,M)}(t)=\P_E(t)\,\I^M(t)$.
Comparing the coefficients of $t^{\depth M}$ in $\I^M(t)$ and $\P_E(t)\,\I^M(t)$, we observe that $E$ is cyclic.
Thus the ring $R$ is Gorenstein.

(7) $\Rightarrow$ (5) and (6):
By the argument in the proof of the implication (7) $\Rightarrow$ (1) and (2), we have $\t(R)\cong\e(R)\cong R$.
Hence the $R$-modules $\t(R)$ and $\e(R)$ are nonzero and free.

(5) $\Rightarrow$ (3) and (6) $\Rightarrow$ (4):
Letting $M=R$, we are done.

(3) $\Rightarrow$ (7):
As $\t(M)$ has finite G-dimension, it is in $\A(R)$ by \cite[(3.1.11)]{C}.
Since $\t(M)\cong\Hom_R(E,M)$ by Proposition \ref{2}, we have $M\in\B(R)$ by \cite[Theorem 2.8(a)]{cdim}.
Since $M$ has finite projective dimension, it follows from \cite[(3.4.12)]{C} that $R$ is Gorenstein.

(4) $\Rightarrow$ (7):
As $M$ has finite G-dimension, it is in $\A(R)$ by \cite[(3.1.11)]{C}.
Hence $\e(M)\cong E\otimes M\in\B(R)$ by Proposition \ref{2} and \cite[Theorem 2.8(b)]{cdim}.
Since $\e(M)$ has finite projective dimension, it follows from \cite[(3.4.12)]{C} that $R$ is Gorenstein.
\end{proof}

To state our next result, we recall a basic fact on faithful modules.

\begin{rem}\label{12}
An $R$-module $X$ is faithful if and only if the homothety map $R\to\Hom_R(X,X)$ is injective.
Indeed, the kernel of this map coincides with the annihilator of $X$.
\end{rem}

The following lemma is necessary to get our next corollary.

\begin{lem}\label{4}
Let $M$ be a finitely generated $R$-module such that $M\cong\t(M)$.
\begin{enumerate}[\rm(1)]
\item 
If $M$ is faithful, then so is $E$.
\item
If $E$ is faithful, then $E_\p\cong R_\p$ for every $\p\in\Ass_RM$.
\end{enumerate}
\end{lem}

\begin{proof}
(1) Proposition \ref{2} yields $\t(M)\cong\Hom_R(E,M)$.
We get $\ann_RE\subseteq\ann_R\t(M)\subseteq\ann_RM$, where the latter inclusion follows from the assumption that $M\cong\t(M)$.
Thus, the assertion holds.

(2) Since $M\cong\t(M)\cong\Hom_R(E,M)$, we have $\p\in\Ass M\subseteq\Supp M\subseteq\Supp E$.
In view of Propositions \ref{2} and \ref{3}, there are isomorphisms
\begin{align*}
&M_\p\cong\t_R(M)_\p\cong\Hom_R(E,M)_\p\cong\Hom_{R_\p}(E_\p,M_\p)\cong\Hom_{R_\p}(\c_{R_\p},M_\p)\cong\t_{R_\p}(M_\p),\\
&\Hom_{R_\p}(\kappa(\p),M_\p)\cong\Hom_{R_\p}(\kappa(\p),\Hom_{R_\p}(E_\p,M_\p))\cong\Hom_{R_\p}(E_\p\otimes_{R_\p}\kappa(\p),M_\p).
\end{align*}
Since $\p$ is in $\Ass M$, the $\kappa(\p)$-vector space $\Hom_{R_\p}(\kappa(\p),M_\p)$ is nonzero.
Substituting $\kappa(\p)^{\oplus c}$ for $E_\p\otimes_{R_\p}\kappa(\p)$, we see that $c=1$, so that $E_\p\otimes_{R_\p}\kappa(\p)\cong\kappa(\p)$.
Nakayama's lemma implies that $E_\p$ is cyclic over $R_\p$.
Faithfulness localizes by Remark \ref{12}, and $E_\p$ is a faithful $R_\p$-module.
It follows that $E_\p\cong R_\p$.
\end{proof}

Now we can prove another corollary of Theorem \ref{8}.

\begin{cor}\label{13}
Let $M$ be a finitely generated $R$-module such that $M\cong\t(M)$.
Consider the following two conditions.
\begin{enumerate}[\qquad\rm(i)]
\item
The $R$-module $M$ is faithful.
\item 
The $R$-module $E$ is faithful and $\dim M=\dim R$.
\end{enumerate}
Then the following statements hold true.
\begin{enumerate}[\rm(1)]
\item
If {\rm(i)} holds, then {\rm(ii)} holds as well.
\item
If {\rm(ii)} holds, then $R$ is a Cohen--Macaulay local ring and $E$ is a canonical module of $R$.
\item 
If {\rm(ii)} holds, then the ring $R$ is locally Gorenstein on $\Ass M$.
\item
If {\rm(i)} holds, then the ring $R$ is generically Gorenstein.
\end{enumerate}
\end{cor}

\begin{proof}
(1) The equality $\ann M=0$ particularly says $\dim M=\dim R$. Lemma \ref{4}(1) implies that $E$ is faithful as well.

(2) Since $\dim M=\dim R$, there exists $\p\in\Assh R\cap\Supp M$.
Then $\p\in\Min M\subseteq\Ass M$.
Lemma \ref{4}(2) implies $E_\p\cong R_\p$, and hence $\dim E=\dim R$.
As $E\cong\e(R)$ by Proposition \ref{2}, it follows from Theorem \ref{8} and Proposition \ref{15}(3) that $R$ is Cohen--Macaulay and $E$ is a canonical module of $R$.

(3) Let $\p\in\Ass M$.
By (2), the local ring $R_\p$ is Cohen--Macaulay, and $E_\p$ is a canonical module of $R_\p$.
It follows from Lemma \ref{4}(2) that $E_\p$ is isomorphic to $R_\p$, which means that $R_\p$ is Gorenstein.

(4) Since the $R$-module $M$ is faithful, $R$ is embedded in a finite direct sum of copies of $M$.
This implies that $\Ass R$ is contained in $\Ass M$.
The assertion is now a direct consequence of (3) and (1).
\end{proof}

\section{Proofs of our main results}

This section is devoted to giving proofs of the results stated in the Introduction and stating related results.

\begin{proof}[Proof of Theorem \ref{14}]
(1) A quasi-canonical module $\c_R$ of $R$ exists by Proposition \ref{15}(1).
Set $X=\Hom_R(\c_R,M)$ and $Y=\c_R\otimes_RM$.
Let $Q\twoheadrightarrow\widehat R$ be a Cohen presentation.
Applying Proposition \ref{2} and Corollary \ref{16}, we have
$$
\widehat X\cong\Hom_{\widehat R}(\widehat{\c_R},\widehat M)\cong\Hom_{\widehat R}(\c_{\widehat R},\widehat M)\cong\t_{\widehat R}(\widehat M)=\Tor_{\dim Q-\depth R}^Q(\widehat R,\widehat M),
$$
and similarly $\widehat Y\cong\Ext_Q^{\dim Q-\depth R}(\widehat R,\widehat M)$.
Therefore, $X$ is a characteristic module of $M$, and $Y$ is a cocharacteristic module of $M$.

(2) Suppose that $\T_M$ exists. Then there exists a Cohen presentation $Q\twoheadrightarrow\widehat R$ such that $\widehat{\T_M}\cong\Tor_{\dim Q-\depth R}^Q(\widehat R,\widehat M)$.
Using Proposition \ref{2}, we have $
\widehat{\T_M}\cong\t_{\widehat R}(\widehat M)\cong\Hom_{\widehat R}(\c_{\widehat R},\widehat M)$.
Hence $\T_M$ is uniquely determined from $M$ up to isomorphism.
Similarly, $\E_M$ is uniquely determined from $M$ up to isomorphism, if it exists.
\end{proof}

\begin{proof}[Proof of Theorem \ref{17}]
Using Propositions \ref{15}, \ref{2} and Theorem \ref{14}, we see that $\c_{\widehat R}$ and $\E_{\widehat R}$ exist and $\c_{\widehat R}\cong\E_{\widehat R}$.

We prove (1) and (2) together. Suppose that $\E_R$ exists.
By definition, there exists a Cohen presentation $P\twoheadrightarrow\widehat R$ such that $\widehat{\E_R}\cong\Ext_P^{\dim P-\depth R}(\widehat R,\widehat R)$.
Using Propositions \ref{15}(2), \ref{2} and Theorem \ref{14}(2), we get
$$
\Ext_P^{\dim P-\depth R}(\widehat R,P)
\cong\c_{\widehat R}
\cong\e_{\widehat R}(\widehat R)
\cong\Ext_P^{\dim P-\depth R}(\widehat R,\widehat R)
\cong\E_{\widehat R}.
$$
Hence $\E_R$ is a quasi-canonical module of $R$, and we have $\c_R\cong\E_R$ by Proposition \ref{15}(2). 
If $R$ is Cohen--Macaulay, then $\E_R$ is a canonical module of $R$ by Proposition \ref{15}(3).

Let $Q\twoheadrightarrow\widehat R$ be a Cohen presentation.
By Propositions \ref{15}(2) and\ref{2}, we get $
\c_{\widehat R}\cong\Ext_Q^{\dim Q-\depth R}(\widehat R,Q)\cong\Ext_Q^{\dim Q-\depth R}(\widehat R,\widehat R)$.
Hence $\E_{\widehat R}\cong\widehat{\E_R}\cong\Ext_Q^{\dim Q-\depth R}(\widehat R,\widehat R)\cong\Ext_Q^{\dim Q-\depth R}(\widehat R,Q)$.

(3) Suppose that $\T_M$ exists.
Then $\widehat{\T_M}\cong\Tor_{\dim Q-\depth R}^Q(\widehat R,\widehat M)$ for some Cohen presentation $Q\twoheadrightarrow\widehat R$.
Theorem \ref{14} implies that $\T_{\widehat M}$ exists and is isomorphic to $\Tor_{\dim Q-\depth R}^Q(\widehat R,\widehat M)$.
Hence $\widehat{\T_M}\cong\T_{\widehat M}$.
Proposition \ref{2} shows that $\T_{\widehat M}\cong\Hom_{\widehat R}(\c_{\widehat R},\widehat M)\cong\Hom_{\widehat R}(\E_{\widehat R},\widehat M)$.
In a similar way, we see that if $\E_M$ exists, then $\E_{\widehat M}\cong\widehat{\E_M}\cong\E_{\widehat R}\otimes_{\widehat R}\widehat M$.
\end{proof}

\begin{proof}[Proof of Theorem \ref{9}]
Passing to the completion, we may assume that $R$ is complete.
Then the assertion follows from Theorem \ref{8}.
\end{proof}

\begin{proof}[Proof of Theorem \ref{10}]
By completion, we may assume that $R$ is complete, and then we are done by Corollary \ref{11}.
\end{proof}

From now on, we shall state several other results than the ones given in the Introduction.
We begin with showing that the cocharacteristic module of a local ring $R$ and the quasi-canonical module of $R$ are the same.

\begin{cor}
Let $R$ be a local ring.
Suppose that either $\c_R$ or $\E_R$ exists.
Then the other exists as well, and both are isomorphic: one has $\c_R\cong\E_R$.
\end{cor}

\begin{proof}
We have already seen in the proof of Theorem \ref{17} that if $\E_R$ exists, then $\c_R$ also exists and $\E_R\cong\c_R$. Suppose
 that $\c_R$ exists.
Then $\widehat{\c_R}\cong\c_{\widehat R}\cong\e_{\widehat R}(\widehat R)=\Ext_Q^{\dim Q-\depth R}(\widehat R,\widehat R)$ for some Cohen presentation $Q\twoheadrightarrow\widehat R$, where the two isomorphisms follow from Corollary \ref{16} and Proposition \ref{2}, respectively.
By the definition of a cocharacteristic module, we see that $\c_R$ is a cocharacteristic module of $R$.
It follows from Theorem \ref{14}(2) that $\c_R$ is isomorphic to $\E_R$.
\end{proof}

Finally, we remove the assumption from Corollary \ref{13}(3)(4) that $R$ is a homomorphic image of a regular local ring.
We do this after stating a remark.

\begin{rem}\label{18}
Let $R$ be a local ring.
For a finitely generated $R$-module $M$ it holds that
$$
\Ass_RM=\{P\cap R\mid P\in\Ass_{\widehat R}\widehat M\}.
$$
Indeed, let $\p\in\Ass_RM$.
Then $\Ass_{\widehat R}\widehat M=\bigcup_{\q\in\Ass_RM}\Ass_{\widehat R}\widehat R/\q\widehat R$ by \cite[Theorem 12(ii)]{M}, which contains $\Ass_{\widehat R}\widehat R/\p\widehat R$.
Since $\widehat R/\p\widehat R$ is nonzero, there exists $P\in\Ass_{\widehat R}\widehat R/\p\widehat R\subseteq\Ass_{\widehat R}\widehat M$.
Then $P\cap R\in\{P'\cap R\mid P'\in\Ass_{\widehat R}\widehat R/\p\widehat R\}=\{\p\}$ by \cite[Theorem 12(i)]{M}, and hence $P\cap R=\p$.

Conversely, let $P\in\Ass_{\widehat R}\widehat M$ and put $\p=P\cap R$.
Then $P$ is in $\bigcup_{\q\in\Ass_RM}\Ass_{\widehat R}\widehat R/\q\widehat R$, so that $P\in\Ass_{\widehat R}\widehat R/\q\widehat R$ for some $\q\in\Ass_RM$.
It follows that $\p=P\cap R\in\{P'\cap R\mid P'\in\Ass_{\widehat R}\widehat R/\q\widehat R\}=\{\q\}$, which implies that $\p=\q\in\Ass_RM$.
\end{rem}

\begin{cor}
Let $R$ be a local ring and $M$ a finitely generated $R$-module.
Suppose that $\E_R$ and $\T_M$ exist and that $\T_M\cong M$.
\begin{enumerate}[\rm(1)]
\item
If $\E_R$ is faithful and $\dim M=\dim R$, then $R$ is locally Gorenstein on $\Ass_RM$.
\item 
If $M$ is faithful, then $R$ is generically Gorenstein.
\end{enumerate}
\end{cor}

\begin{proof}
There are isomorphisms $\widehat M\cong\widehat{\T_M}\cong\t_{\widehat R}(\widehat M)$.

(1) By Remark \ref{12}, the $\widehat R$-module $\e_{\widehat R}(\widehat R)=\widehat{\E_R}$ is faithful.
Also, $\dim\widehat M=\dim\widehat R$.
If $\widehat R$ is locally Gorenstein on $\Ass_{\widehat R}\widehat M$, then $R$ is locally Gorenstein on $\Ass_RM$ by Remark \ref{18}.
Thus the assertion follows from Corollary \ref{13}(3).

(2) By Remark \ref{12}, the $\widehat{R}$-module $\widehat{M}$ is faithful.
Next, we show that if $\widehat{R}$ is generically Gorenstein, then so is $R$. Indeed, if $\p \in \Ass R$, then by Remark \ref{18} there exists $P\in \widehat{R}$ such that $\p=P \cap R$. Then the map $R_{\p}\to \widehat{R}_P$ is flat by \cite[Theorem 7.1]{M2}, and by \cite[Theorem 23.4]{M2}, $R_\p$ is Gorenstein. Now the assertion follows from Corollary \ref{13}(4).
\end{proof}

We close the section by posing a natural question, asking whether one can remove from Theorem \ref{10}(5) the assumption that $M$ has finite G-dimension.

\begin{ques}\label{q2}
Let $R$ be a local ring.
Suppose that there is a finitely generated $R$-module $M$ with $\dim M=\dim R$ such that $\T_M$ exists and $\T_M\cong M$.
Is then $R$ Gorenstein?
\end{ques}

Corollary \ref{artinian} shows that this question is affirmative in the case when the local ring $R$ is artinian. After first draft of the paper appeared on arXiv, Dey and Ghosh \cite{DG} addressed Question \ref{q2} negatively in general; see \cite[Proposition 3.13]{DG}

\begin{ac}
Part of this work was done during Takahashi's visit to the Florida A\&M University in February, 2024. He is grateful to the members of the Department of Mathematics for their hospitality. The authors thank the anonymous referee very much for carefully checking arguments and for his/her comments that significantly improved the paper.
\end{ac}

\end{document}